\documentclass[11pt]{amsart}
\usepackage{color}
\usepackage{graphicx}
\usepackage{microtype}
\usepackage[colorlinks]{hyperref}

\newtheorem{thm}{Theorem}[section]
\newtheorem{lem}[thm]{Lemma}
\newtheorem{prop}[thm]{Proposition}

\theoremstyle{definition}
\newtheorem{defn}[thm]{Definition}

\theoremstyle{definition}
\newtheorem{ex}[thm]{Example}
\newtheorem{rmk}[thm]{Remark}

\newcommand\cO{\mathcal O}
\newcommand\defi[1]{\emph{#1}}

\newcommand\disjoint{\amalg}
\newcommand\FF{\mathbb F}
\newcommand\Gal{\operatorname{Gal}}

\newcommand\isom{\cong}

\newcommand\PGL{\operatorname{PGL}}
\newcommand\PP{\mathbb P}

\newcommand\Spec{\operatorname{Spec}}

\newcommand\X{\mathfrak X}
\newcommand\ZZ{\mathbb Z}

\title{Lifting matroid divisors on tropical curves}
\author{Dustin Cartwright}
\address{Department of Mathematics \\ University of Tennessee \\
    227 Ayres Hall \\ Knoxville, TN 37996-1320}
\email{cartwright@utk.edu}

\begin{document}
\begin{abstract}
Tropical geometry gives a bound on the ranks of divisors on curves in terms
of the combinatorics of the dual graph of a degeneration. We show that for a
family of examples, curves realizing this bound might only exist over certain
characteristics or over certain fields of definition. Our examples also apply to
the theory of metrized complexes and weighted graphs. These examples arise by
relating the lifting problem to matroid realizability. We also give a proof of
Mn\"ev universality with explicit bounds on the size of the matroid, which may
be of independent interest. 
\end{abstract}

\maketitle

\section{Introduction}\label{sec:introduction}

The specialization inequality in tropical geometry gives an upper bound for the
rank of a divisor on a curve in terms of a combinatorial quantity known as the
rank of the specialization of the divisor on the dual graph of the special fiber
of a degeneration~\cite{baker}. This bound can be sharpened by incorporating
additional information about the components of the special fiber, giving
augmented graphs~\cite{amini-caporaso} or metrized complexes~\cite{amini-baker}.
All of these inequalities can be strict because there may be many algebraic
curves and divisors with the same specialization. Thus, the natural question is
whether, for a given graph and divisor on that graph, the inequality is sharp
for \emph{some} algebraic curve and divisor. If $R$ is the discrete valuation
ring over which the degeneration of the curve is defined, we will refer to such
a curve and divisor
as a \defi{lifting} of the graph with its divisor over~$R$.
In this paper, we show that the
existence of a lifting can depend strongly on the characteristic of the field:

\begin{thm}\label{thm:characteristic}
Let $P$ be any finite set of prime numbers. Then there exist graphs $\Gamma$ and
$\Gamma'$ with rank~$2$ divisors $D$ on~$\Gamma$ and~$D'$ on~$\Gamma'$ with the
following property:
For any infinite field~$k$, $\Gamma$ and~$D$ lift over $k[[t]]$ if and only
if the characteristic of~$k$ is in~$P$, and $\Gamma'$ and~$D'$ lift over
$k[[t]]$ if and only if the characteristic of~$k$ is not in~$P$.
\end{thm}

We also show that the existence of a lift depends on the field even beyond
characteristic:

\begin{thm}\label{thm:number-field}
Let $k'$ be any number field. Then there exists a graph~$\Gamma$ with a rank~$2$
divisor~$D$ such that for any field~$k$ of characteristic~$0$, $\Gamma$ and~$D$
lift over~$k[[t]]$ if and only if $k$ contains $k'$.
\end{thm}

Both Theorem~\ref{thm:characteristic} and~\ref{thm:number-field} are
immediate consequences of the following:

\begin{thm}\label{thm:universality}
Let $X$ be a scheme of finite type over $\Spec \ZZ$. Then there exists a graph
$\Gamma$ with a rank~$2$ divisor~$D$ such that, for any infinite field $k$,
$\Gamma$ and~$D$ lift over $k[[t]]$ if and only if $X$ has a $k$-point.
\end{thm}

Theorems~\ref{thm:characteristic}, \ref{thm:number-field},
and~\ref{thm:universality} all apply equally well to divisors on weighted
graphs~\cite{amini-caporaso}
because the construction of a degeneration in Theorem~\ref{thm:universality}
uses curves of genus~$0$ in the special fiber and for such components, the
theory of weighted graphs agrees with unweighted graphs.

Moreover, these theorems also apply to the metrized complexes introduced
in~\cite{amini-baker}, which record the isomorphism types of the curves in the
special fiber. Again, for rational components, the rank of the metrized complex
will be the same as the rank for the underlying graph. For metrized complexes,
there is a more refined notion of a limit~$g^r_d$, which involves additionally
specifying vector spaces of rational functions at each vertex. Not every divisor
of degree~$d$ and rank~$r$ on a metrized complex lifts to a limit~$g^r_d$, but
the examples from the above theorems do:
\begin{prop}\label{prop:limit-grd}
Let $\Gamma$ and~$D$ be a graph and divisor constructed as in
Theorem~\ref{thm:universality}. Then for any lift of $\Gamma$ to a metrized
complex with rational components, there also exists a lift of~$D$ to a
limit~$g^2_d$.
\end{prop}

If we were to consider divisors of rank~$1$ rather than rank~$2$, \cite{abbr2}
provides a general theory for lifting. They prove that if a rank~$1$ divisor can
be lifted to a tame harmonic morphism with target a genus~$0$ metrized complex,
then it lifts to a rank~$1$ divisor an algebraic curve. Moreover, the converse
is true except for possibly some cases of wild ramification in positive
characteristic. Using this, they give examples of rank~$1$ divisors which do not
lift over any discrete valuation ring~\cite[Sec.~5]{abbr2}. While the existence
of a tame harmonic morphism depends on the characteristic, the dependence is
only when the characteristic is at most the degree of the
divisor~\cite[Rmk.~3.9]{abbr2}. In contrast, lifting rank~$2$ divisors can
depend on the characteristic even when the characteristic is bigger than the
degree:
\begin{thm}\label{thm:quantitative-characteristic}
If $P = \{p\}$ where $p \geq 443$ is prime, then the divisors~$D$ and~$D'$ in
Theorem~\ref{thm:characteristic} can be taken to have degree less than~$p$.
\end{thm}

For simplicity, we have stated Theorems~\ref{thm:characteristic},
\ref{thm:number-field}, and~\ref{thm:universality} in terms of liftings over
rings of formal power series, but some of our results also apply to other discrete
valuation rings. In particular, these theorems apply verbatim with $k[[t]]$
replaced by any DVR which contains its residue field~$k$. For other, possibly
even mixed characteristic DVRs, we have separate necessary and sufficient
conditions in Theorems~\ref{thm:matroid-lifting} and~\ref{thm:construction}
respectively.

The proof of Theorem~\ref{thm:universality} and its consequences use Mn\"ev's
universality theorem for matroids~\cite{mnev}. Matroids are combinatorial
abstractions of vector configurations in linear algebra. However, not all
matroids come from vector configurations and those that do are called
realizable. Mn\"ev proved that realizability problems for rank~$3$ matroids in
characteristic~$0$ can encode arbitrary systems of integral polynomial equations
and Lafforgue extended this to arbitrary
characteristic~\cite[Thm.~1.14]{lafforgue}. Thus, Theorem~\ref{thm:universality}
follows from universality for matroids together with a connection between
matroid realizability and lifting problems, which is done in
Theorems~\ref{thm:matroid-lifting} and~\ref{thm:construction}. We also give a
proof of universality in arbitrary characteristic with explicit bounds on the
size of the matroid in order to establish
Theorem~\ref{thm:quantitative-characteristic}.

Matroids have appeared before in tropical geometry and especially as
obstructions for lifting. For example, matroids yield examples of
matrices whose Kapranov rank exceeds their tropical rank, showing that the
minors do not form a tropical basis~\cite[Sec.~7]{develin-santos-sturmfels}. In
addition, Ardila and Klivans defined the tropical linear space
for any simple matroid, which generalizes the tropicalization of a linear
space~\cite{ardila-klivans}. The tropical linear spaces are
realizable as the tropicalization of an algebraic variety if and only if the
matroid is realizable~\cite[Cor.~1.5]{katz-payne}. This paper is only concerned
with rank~$3$ matroids, which correspond to $2$-dimensional fans and the
graphs for which we construct lifting obstructions are the links of the fine
subdivision of the tropical linear space (the fine subdivision is defined in
\cite[Sec.~3]{ardila-klivans}).

Moreover, the matroid divisors from this paper have found applications to other
questions regarding the divisor theory of graphs. David Jensen has shown that
the matroid divisor of the Fano matroid gives an example of a $2$-connected
graph which is not Brill-Noether general for any metric
parameters~\cite{jensen}. In addition to the Baker-Norine rank of a divisor used
in this paper, Caporaso has given a definition of the \defi{algebraic rank} of a
divisor, which involves quantifying over all curves over a given
field~\cite{caporaso}.
Yoav Len has shown that in contrast to the results in
Section~\ref{sec:divisors}, the algebraic rank of a matroid divisor detects
realizability of the matroid, and he has used this to show that the algebraic
rank can depend on the field~\cite{len}.

Since rank~$3$ matroids give obstructions to lifting rank~$2$ divisors on
graphs, it is natural to wonder if higher rank matroids give similar examples
for lifting higher rank divisors.
While we certainly expect there to be results similar to
Theorems~\ref{thm:characteristic}, \ref{thm:number-field},
and~\ref{thm:quantitative-characteristic} for divisors on graphs which have
ranks greater than~$2$, it is not clear that higher rank matroids would provide
such examples, or even what the right encoding of the matroid in a graph would
be. From a combinatorial perspective, our graphs are just order complexes of the
lattice of flats, but for higher rank matroids, the order complex is a
simplicial complex but not a graph.

This paper is organized as follows. In Section~\ref{sec:divisors}, we introduce
the matroid divisors which are our key class of examples and show that as combinatorial
objects they behave as if they should have rank~$2$. In Section~\ref{sec:lifting}, we relate
the lifting of matroid divisors to the realizability of the matroid.
Section~\ref{sec:brill-noether} looks at the applicability of our matroid to the
question of lifting tropically Brill-Noether general divisors and shows that,
with a few exceptions,
matroid divisors are not Brill-Noether general. Finally,
Section~\ref{sec:mnev} provides a quantitative proof of Mn\"ev universality as the
basis for Theorem~\ref{thm:quantitative-characteristic}.

\subsection*{Acknowledgments}
I would like to thank Spencer Backman, Melody Chan, Alex Fink, Eric Katz, Yoav
Len, Diane Maclagan, Sam Payne, Kristin Shaw, and Ravi Vakil for helpful
comments on this project. The
project was begun while the author was supported by NSF grant DMS-1103856L.

\section{Matroid divisors}\label{sec:divisors}

In this section, we construct the divisors and graphs that are used in
Theorem~\ref{thm:universality}. As in~\cite{baker} and~\cite{baker-norine-rr},
we will refer to a finite formal sum of the vertices of a graph as a
\defi{divisor} on that graph. Divisors are related by so-called ``chip-firing
moves'' in which the weight at one vertex is decreased by its degree and those
of its neighbors are correspondingly each increased by~$1$. A reverse
chip-firing move is the inverse operation.

As explained in the introduction, the starting point in our construction is a
rank~$3$ simple matroid. A matroid is a combinatorial model for an arrangement
of vectors, called elements, in a vector space. A rank~$3$ simple matroid
corresponds to such an arrangement in a $3$-dimensional vector space, for which
no two vectors are multiples of each other. There are many equivalent
descriptions of a matroid, but we will work with the flats, which correspond to
vector spaces spanned by subsets of the arrangement, and are identified with the
set of vectors that they contain. For a rank~$3$ simple matroid, there is only
one rank~$0$ and one rank~$3$ flat, and the rank~$1$ flats correspond to the
elements of the matroid, so our primary interest will be in rank~$2$ flats.
Throughout this paper, \defi{flat} will always refer to a rank~$2$ flat.

We refer the reader to~\cite{oxley} for a thorough reference on matroid theory,
or~\cite{katz} for an introduction aimed at algebraic geometers. However, in the
case of interest for this paper, we can give the following axiomatization:
\begin{defn}
A \emph{rank 3 simple matroid}~$M$ consists of a finite set~$E$ of
\emph{elements} and a collection~$F$ of subsets of~$E$, called the \emph{flats}
of~$M$, such that any pair of elements is contained in exactly one flat, and
such that there are at least two flats.
A \defi{basis} of such a matroid is a triple of elements which
are not all contained in a single flat.
\end{defn}

By projectivizing the vector configurations above, a configuration of distinct
$k$-points in the projective plane~$\PP^2_k$ determines a matroid. The
elements of this matroid are the points of the configuration and the flats
correspond to lines in~$\PP^2_k$, identified with the points contained in them.
A matroid coming from a point configuration in this way is called
\defi{realizable over~$k$} and in Section~\ref{sec:lifting}, we will use the
fact that matroid realizability can depend on the field.

Given a rank~$3$ simple matroid~$M$ with elements~$E$ and flats~$F$, we let
$\Gamma_M$ be the bipartite graph with vertex set $E \disjoint F$, and an edge
between $e \in E$ and $f \in F$ when $e$ is contained in~$f$. The graph
$\Gamma_M$ is sometimes called the Levi graph of~$M$. We let $D_M$ be the
divisor on the graph~$\Gamma_M$ consisting of the sum of all vertices
corresponding to elements of the ground set~$E$.

\begin{prop}\label{prop:rank}
The divisor~$D_M$ has rank~$2$.
\end{prop}

\begin{proof}
To prove the proposition, we first need to show that for any degree~$2$
effective divisor~$E$, the difference $D_M - E$ is linearly equivalent to an
effective divisor. We build up a ``toolkit'' of divisors linearly equivalent
to~$D_M$. First, for any flat~$f$, we can reverse fire $f$. This moves a chip
from each element contained in~$f$ to~$f$ itself. Thus, the result is an
effective divisor whose multiplicity at $f$ is the cardinality of~$f$, which is
at least~$2$. Our second chip-firing move is to reverse fire a vertex $e$ as
well as all flats containing $e$. The net effect will be no change at $e$ but
each neighbor~$f$ of~$e$ will end with $\lvert f\rvert - 1 \geq 1$ chips. Third, we
will use the second chip-firing move, after which all the flats which
contain~$e$ have at least one chip, after which it is possible to reverse
fire~$e$ again.

Now let $E$ be any effective degree~2 divisor on~$\Gamma_M$. Thus, $E$ is the
sum of two vertices of~$\Gamma_M$. We consider the various combinations which
are possible for these vertices. First, if $E = [e] + [e']$ for distinct
elements $e$ and~$e'$, then $\Gamma_M - E$ is effective. Second, if $E =
[e] + [f]$, then we have two subcases. If $e$ is in $f$, then we reverse
fire $e$ and all flats containing it. If $e$ is not in~$f$, then we can
reverse fire just $f$. Third, if $E = [f] + [f']$ for distinct flats $f$
and~$f'$, then there are again two subcases. If $f$ and $f'$ have no
elements in common, then we can reverse fire $f$ and~$f'$. If $f$ and~$f'$
have a common element, say $e$, then we reverse fire $e$ together with the
flats which contain it. Fourth, if $E = 2[e]$, then we use the third
chip-firing move, which will move one chip onto~$e$ for each flat
containing~$e$, of which there are at least~$2$.
Fifth, if $E = 2[f]$, then we reverse fire $f$.

Finally, to show that the rank is at most~$2$, we give an effective degree~3
divisor~$E$ such that $D_M - E$ is not linearly equivalent to any effective
divisor. For this, let $e_1$, $e_2$, and~$e_3$ form a basis for~$M$ and let
$f_{ij}$ be the unique flat containing $e_i$ and~$e_j$ for $1 \leq i < j \leq
3$. We set $E = [f_{12}] + [f_{13}] + [f_{23}]$ and claim that $D_M - E$ is not
linearly equivalent to any effective divisor. We reverse fire $e_1$ together
with all flats containing it to get the following divisor linearly
equivalent to $D_M - E$:
\begin{equation}\label{eq:reduced-divisor}
[e_1] + \big(\lvert f_{12}\rvert - 2\big)[f_{12}] +
\big(\lvert f_{13} \rvert - 2\big)[f_{13}] -
[f_{23}]
+ \sum_{\genfrac{}{}{0pt}{1}{f_k \ni e_1}{f_k \neq f_{12},f_{13}}}
\big(\lvert f_{k} \rvert - 1\big) [f_k],
\end{equation}
which is effective except at $f_{23}$.

We wish to show the divisor in~(1) is not linearly equivalent to any effective
divisor, which we will do by showing that it is $f_{23}$-reduced using Dhar's
burning algorithm~\cite{dhar}. We first claim that for any element $e$ other
than~$e_1$, there is a path from $f_{23}$ to $e$ which does not encounter any
chips. If $e$ is in $f_{23}$, then there is a direct edge between these
vertices. If not, then we first let $f$ denote the unique flat containing both
$e$ and $e_1$. Since $e_1$, $e_2$, and $e_3$ form a basis, they cannot all be
contained in~$f$. Without loss of generality, we can assume that $e_2$ is not
in~$f$, and so $e$, $e_1$, and $e_2$ form a basis. Thus, if $f'$ is the unique
flat containing $e_2$ and $e$, then $f'$ does not contain $e_1$. Therefore, the
path from $f_{23}$ to~$e_2$ to $f'$ to $e$ does not cross any chips.

\begin{figure}
\includegraphics{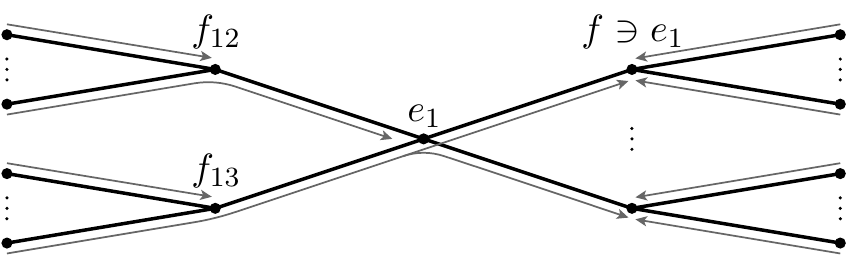}
\caption{Paths through the graph $\Gamma_M$ taken by the burning algorithm
applied to show the divisor~(\ref{eq:reduced-divisor}) is $f_{23}$-reduced.
There is a path from $f_{23}$ to any element other than $e_1$, which are
depicted on the left and right, and from these vertices, there are paths which
eventually equal or exceed the number of chips at all vertices of the
graph.}\label{fig:burning}
\end{figure}
By the claim in the previous paragraph, the burning algorithm will lead to
$\lvert f_{12} \rvert -1$ independent ``fires'' arrviving at $f_{12}$, one for
each element in $f_{12} \setminus e_1$. Thus, these remove the $\lvert f_{12}
\rvert -2$ chips on $f_{12}$ and one path continues on to remove the single chip
from $e_1$. Likewise, $\lvert f_{13} \rvert - 1$ independent ``fires'' arive at
$\lvert f_{13} \rvert - 2$, and one continues on and passes through $e_1$ to
arrive at all the flats containing $e_1$. Therefore, for every flat~$f$ that
contains $e_1$, other than $f_{12}$ and $f_{13}$, which have already been
handled, there is a ``fire'' arriving from every element of~$f$, which exceeds
the $\lvert f \rvert - 1$ chips on this vertex. The paths used to cover
$\Gamma_M$ are summarized schematically in
Figure~\ref{fig:burning}. Since the burning algorithm
covers the graph $\Gamma_M$, we conclude that the divisor $D_M - E$ is
$f_{23}$-reduced and so not linearly equivalent to an effective divisor.
\end{proof}

Proposition~\ref{prop:rank} also shows that if $\Gamma_M$ is made into a
weighted graph by giving all vertices genus~$0$, then $D_M$ has rank~$2$ on the
weighted graph. The rank is, again, unchanged for any lifting of the weighted
graph to a metrized complex. To show that $D_M$ is also a limit $g^2_d$ as in
Proposition~\ref{prop:limit-grd},
we also need to choose $3$-dimensional vector spaces of rational functions on
the variety attached to each vertex.

\begin{proof}[Proof of Proposition~\ref{prop:limit-grd}]
We recall from~\cite{amini-baker} that a lift of~$\Gamma_M$ to a metrized
complex means associating a $\PP^1_k$ for each vertex~$v$ of the graph, which we
denote $C_v$, and a point on~$C_v$ for each edge incident to~$v$. A lift of the
divisor~$D_M$ is a choice of a point on $C_e$ for each element~$e$ of~$M$.

The data of a limit $g^2_d$ is a $3$-dimensional vector space~$H_v$ of rational
functions on each~$C_v$~\cite[Sec.~5]{amini-baker}, which we choose as follows. For each flat~$f$, we
arbitrarily choose two elements from it and let $p_{f,1}$
and~$p_{f,2}$ be the points on~$C_f$ corresponding to the edges from $f$ to each of the
chosen elements. Our vector
space~$H_{f}$
consists of the rational functions which have at worst simple poles at
$p_{f,1}$ and~$p_{f,2}$. For each element~$e$, we choose an arbitrary flat
containing~$e$ and let $q_{e}$ be the point on~$C_e$ corresponding to the edge
to~$e$. Our
vector space~$H_e$ consists of the rational functions which have at worst
poles at~$q_e$ and at the point of the lift of~$D_M$.

Now to check that these vector spaces form a limit~$g^2_d$, we need to show that
the refined rank is~$2$. For this, we use the same ``toolkit'' functions as in
the proof of Proposition~\ref{prop:rank}, but we augment them with rational
functions from the prescribed vector spaces on the algebraic curves. The first
item from our toolkit was reverse firing a flat~$f$ to produce at least two
points on~$C_{f}$. We can use rational functions with poles at $p_{f,1}$ and
$p_{f,2}$ to produce any degree two effective divisor on $C_{f}$.
For each $C_{e}$ such that $e$ is an element of~$f$, we need to use a rational
function with a zero at the edge to~$f$ and a pole at the lift of the
divisor~$D_M$.

The second item we needed in our toolkit was reverse firing an element~$e$
together with all of the flats which contain it. Here, for each element $e'$
other than~$e$, we use the rational function with a pole at the divisor and a
zero at the point corresponding to the edge to the unique flat containing
both~$e'$ and~$e$. At each flat~$f$ containing $e$, we can use any function with
a pole at~$p_{f,i}$, where $i \in \{1, 2\}$ can be chosen to not be the edge
leading to~$e$. This produces a divisor at an arbitrary point of~$C_{e}$.

The third and final operation we used was the previous item followed by a
reverse firing of~$e$. Here, we use the same rational functions as before, but
we can choose any rational function on~$C_e$ which has poles at the point of the
divisor and~$q_e$, thus giving us two arbitrary points on~$C_e$. We conclude
that rational
functions can be found from the prescribed vector spaces to induce a linear
equivalence between the lift of~$D_M$ and any two points on the metrized
complex.
\end{proof}

In the case of rank~$1$ divisors, lifts can be constructed using the theory of
harmonic maps of metrized complexes, which gives a complete theory for divisors
defining tamely
ramified maps to~$\PP^1$~\cite{abbr2}. A sufficient condition for lifting a
rank~$1$ divisor is for it to be the underlying graph of a metrized complex
which has a tame harmonic morphism to a tree (see \cite[Sec.~2]{abbr1}
for precise definitions). These definitions are limited to the rank~$1$ case,
but for rank~$2$ divisors we can subtract points to obtain a divisor of rank at
least~$1$. In particular, if $D_M$ lifts, then for any element~$e$, $D_M - [e]$ will be the
specialization of a rank~$1$ effective divisor. However,
the lifting criterion of~\cite{abbr2} is  satisfied for these subtractions,
independent of the liftability of~$D_M$.

\begin{prop}\label{prop:tame-harmonic}
Let $M$ be any rank~$3$ simple matroid and $e$ any element of~$M$. Also, let $k$
be an algebraically closed field of characteristic not~$2$. Then, $\Gamma_M$
has a tropical modification $\widetilde\Gamma_M$ such that $\widetilde\Gamma_M$
can be lifted to a totally degenerate metrized complex over~$k$ with a tame
harmonic morphism to a genus~$0$ metrized complex, such that one fiber is a lift
of the divisor~$D_M - [e]$.
\end{prop}

\begin{proof}
We first construct a modification~$\widetilde \Gamma_M$ of~$\Gamma_M$ which has
a finite harmonic morphism from $\widetilde \Gamma_M$ to a tree~$T$. The
tree~$T$ will be a star tree with a central vertex~$w$, together with an
unbounded edge, denoted $r_f$, for each flat~$f$ which does not contain~$e$, and
a single unbounded edge~$r_e$ corresponding to~$e$. Our modification
of~$\Gamma_M$ consists of adding the following unbounded edges: At~$e$, we add
one unbounded edge~$s_{e,f}$ for each flat~$f$ containing~$e$. At each
element~$e'$ other than $e$, we add one unbounded edge~$s_{e',f}$ for each
flat~$f$ which contains neither $e$ nor~$e'$. At a flat~$f$, we add unbounded
edges~$s_{f,i}$ where $i$ ranges from $1$ to~$\lvert f\rvert$ if $e \notin f$
and from $1$ to $\lvert f\rvert - 2$ if $e \in f$.

\begin{figure}
\includegraphics{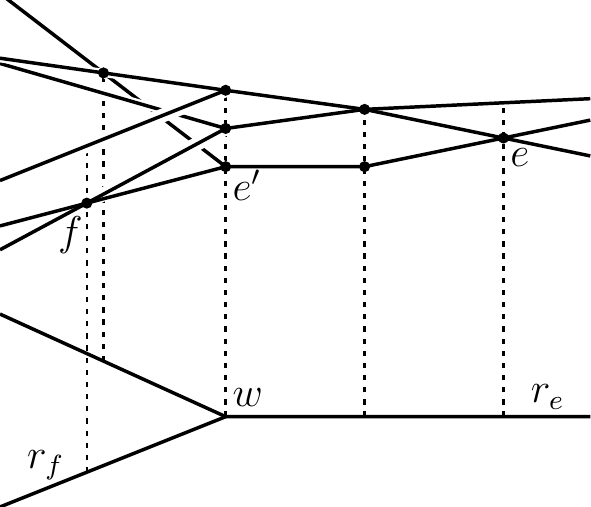}
\caption{Modification of $\Gamma_M$ which has a finite, effective harmonic morphism to the
tree to a tree, such that the
fiber over the central vertex of the tree is $D_M - [e]$. In this figure, $f$ is
a flat which does not contain $e$ and $e'$ is an element of~$f$.}
\label{fig:harmonic}
\end{figure}

We now construct a finite harmonic morphism~$\phi$ from $\widetilde \Gamma_M$
to~$T$. Each element other than $e$ maps to the central vertex~$w$ of~$T$.
Each flat~$f$ not containing~$e$ maps to a point one unit of
distance along the corresponding ray~$r_f$ of~$T$. Then the rays~$s_{e', f}$
and~$s_{f,i}$ also map to the ray $r_f$, starting at~$w$ and $\phi(f)$
respectively.

We map the vertex~$e$ to its unbounded ray~$r_e$, at a distance of~$2$ from~$w$,
which leaves all of the flats containing~$e$ along the same ray at a distance
of~$1$. The rays $s_{e,f}$ and $s_{f,i}$, for flats~$f$ containing~$e$ also
map to~$r_e$, starting distances of~$2$ and~$1$ from~$w$ respectively. The
map~$\phi$ is depicted in Figure~\ref{fig:harmonic}.

To check that $\phi$ is harmonic, we need to verify that locally, around each
vertex~$v$ of $\widetilde \Gamma_M$, the same number of edges map to each of the
edges incident to~$\phi(v)$, and this number is the degree of $\phi$ at
$v$~\cite[Sec.~2]{baker-norine}. First, suppose that the vertex~$v$ corresponds
to an element~$e'$ other than~$e$ and we have defined $\phi(e')$ to be the
central vertex~$w$. In this case, for each ray $r_f$ of~$T$, there is exactly
one edge incident to $e'$ mapping to $r_f$, either the edge between $e'$ and $f$
if $f$ contains $e$, or the unbounded edge $s_{e',f}$ if not. There is also
exactly one edge mapping to $r_e$, which is the edge between $e'$ and the unique
flat containing both $e$ and~$e'$. Therefore, $\phi$ is harmonic at the
vertex~$e'$ with local degree equal to~$1$.

Second, at the vertex~$e$, which maps along the edge $r_e$, there is one edge
mapping to the bounded side of $r_e$ for each flat~$f$ containing $e$ and also
for each such flat, one infinite ray $s_{e,f}$ mapping to the unbounded side
of~$r_e$. Thus $\phi$ is also harmonic at $e$, and has local degree equal to the
number of flats containing~$e$.

Finally, we check that $\phi$ is harmonic at a vertex corresponding to a
flat~$f$. If $f$ does not contain~$e$, then there are $\lvert f \rvert$ rays
mapping to the unbounded side of~$r_f$ and the same number of edges mapping to
the bounded side, connecting $f$ to the elements it contains. Thus, at such a
vertex, $\phi$ is harmonic and its local degree is $\lvert f \rvert$. If $f$
does contain~$e$, then there are $\lvert f \rvert - 2$ rays mapping to the
unbounded side of~$r_e$ together with the edge connecting $f$ to~$e$. On the
bounded side of $r_e$, there are also $\lvert f \rvert - 1$ edges, connecting
$f$ to the elements $f \setminus \{e\}$, and so here $\phi$ is harmonic with
local degree $\lvert f \rvert - 1$.

To lift $\phi$ to a harmonic morphism of totally degenerate metrized complexes,
we need to choose a map $\phi_v \colon \PP^1 \rightarrow \PP^1$ for each
vertex~$v$ of $\widetilde \Gamma_M$ and an identification of the outgoing
directions with points on $\PP^1$. Having assumed characteristic not~$2$, we can
choose a tame homomorphism of degree equal to the degree of~$\phi$ at~$v$
as~$\phi_v$. We identify the edges incident to~$v$ with points of $\PP^1$ at
which $\phi_v$ is unramified, since these edges all have expansion factor equal
to~$1$. Then, the preimage of a $k$-point of the curve at $w$ consists of one
point in each $\PP^1$ corresponding to the elements $e'$ not equal to $e$,
because the local degrees at these vertices are equal to~$1$. Thus, this
preimage is a lift of $D_M \setminus [e]$ and we have our desired morphism of
metrized complexes.
\end{proof}

\section{Lifting matroid divisors}\label{sec:lifting}

In this section, we characterize the existence of lifts of matroid divisors in
terms of realizability of the corresponding matroids. Recall from~\cite{baker},
that if $R$ is a discrete valuation ring with algebraically closed residue
field, then a regular semistable family~$\X$ of curves over~$R$ gives
homomorphism~$\rho$ from the group of divisors on the general fiber to the group
of divisors on the dual graph~$\Gamma$ of the special fiber. This dual graph is
defined to have a vertex~$v$ for each irreducible component of the special fiber
and an edge for each point of intersection between two components. Then, for any
divisor $\widetilde D$ on the general fiber of~$\X$, $\rho(\widetilde D)$ is
defined to be the formal sum of the vertices of~$\Gamma$ with the coefficient of
a vertex~$v$ equal to the degree of the intersection of~$\overline D$ with
$C_v$, where $\overline D$ is the closure of $\widetilde D$ in $\X$ and $C_v$ is
the irreducible component corresponding to~$v$~\cite[Sec.~2A]{baker}. With this
definition, we have an inequality $r(\widetilde D) \leq r(\rho(\widetilde D))$
between the algebraic and graph-theoretic ranks~\cite[Lem.~2.8]{baker}.

We now consider the semistable family~$\X$ over a discrete valuation ring~$R$,
where we drop the assumption that the residue field of~$R$ is algebraically
closed. In this case, we apply the definitions from the previous paragraph by
first base changing to a discretely valued extension $R' \supset R$, such that
the residue field of~$R'$ is algebraically closed and such that a uniformizer
of~$R$ is also a uniformizer of~$R'$. In particular, the dual graph has one
vertex for each geometric irreducible component of the special fiber and it is
independent of the choice of~$R'$. Also, the definition of~$\rho(\widetilde D)$
is independent of~$R'$ because it can be computed by taking the closure of
$\widetilde D$ in $\X$ and recording the degree of the pullback of this Cartier
divisor to each of the geometric irreducible components of the special fiber.
Moreover, for any family~$\X$, there is a finite \'etale extension $R'$ of~$R$
such that, after base changing to~$R'$, the irreducible components of the
special fiber are geometrically irreducible. Therefore, after this base change,
the computation of the dual graph and the specialization map can be carried out
directly on the resulting family over~$R'$. Since the dimension of a linear
system is invariant under base change, we also have a specialization inequality
$r(\widetilde D) \leq r(\rho(\widetilde D))$ on~$\X$.

We will say that a \defi{lifting} over~$R$ of an effective divisor~$D$ of
rank~$r$ on a graph~$\Gamma$ is a regular semistable family~$\X$ over~$R$ with
dual graph~$\Gamma$,
together with an effective divisor~$\widetilde D$ on the general fiber of~$\X$
such that $\rho(\widetilde D) = D$ and $\widetilde D$ has rank~$r$. Here, and
throughout this section, a regular semistable family~$\X$ includes the
hypothesis that $\X$ is semistable after passing to an extension with
algebraically closed residue field. The relationship between liftings of a
matroid divisor~$D_M$ and its matroid depends on the following, slightly weaker
variant of realizability for~$M$:

\begin{defn}
Let $k$ be a field. We say that a matroid~$M$ has a \defi{Galois-invariant
realization over an extension of~$k$} if there exists a finite scheme
in $\PP^2_k$ which becomes a union of distinct points over $\overline
k$, and these points realize~$M$.
\end{defn}

Equivalently, a Galois-invariant realization is a realization over a finite
Galois extension~$k'$ of~$k$ such that the Galois group $\Gal(k'/k)$ permutes
the points of the realization. Thus, the distinction between a realization and a
Galois-invariant realization is only relevant for matroids which have
non-trivial symmetries. Moreover, Lemma~\ref{lem:break-symmetry} will show that
any matroid can be extended to one where these symmetries can be broken, without
affecting realizability over infinite fields.

\begin{ex}
Let $M$ be the matroid determined by all $21$ points of $\PP^2_{\mathbb F_4}$.
Then $M$ is not realizable over $\PP^2_{\mathbb F_2}$ because it contains more
than $7$ elements, and there are only $7$ points in $\PP_{\mathbb F_2}^2$.
However, $M$ is clearly realizable over~$\mathbb F_4$ and the Galois group
$\Gal(\FF_4/\FF_2) \isom \ZZ/2$ acts on these points by swapping pairs. Thus,
$M$ has a Galois-invariant realization over an extension of~$\mathbb F_2$.
\end{ex}

\begin{ex}
Let $M$ be the Hesse matroid of 9~elements and 12 flats. Then $M$ is not
realizable over~$\mathbb R$ by the Sylvester-Gallai theorem. However, the
flex points of any elliptic curve are a realization of~$M$ over~$\mathbb C$.
If the
elliptic curve is defined over~$\mathbb R$, then the set of all flexes
points is also defined over~$\mathbb R$, so $M$ has a Galois-invariant
realization over an extension of~$\mathbb R$.
\end{ex}

\begin{thm}\label{thm:matroid-lifting}
Let $\Gamma_M$ and $D_M$ be the graph and divisor obtained from a rank~$3$
simple matroid~$M$ as in Section~\ref{sec:divisors}. Also, let $R$ be
any discrete valuation ring with residue field~$k$. If $D_M$ lifts
over~$R$, then the matroid~$M$ has a Galois-invariant
realization over an extension of~$k$.
\end{thm}

By projective duality, a point in~$\PP^2$ is equivalent to a line in the
dual projective space~$\PP^2$. Thus, the collection of points realizing a
matroid is equivalent to a collection of lines, in which the flats correspond to
the points of common intersection. It is this dual representation that we will
construct from the lifting.

\begin{proof}[Proof of Theorem~\ref{thm:matroid-lifting}.]
Let $\X$ be the semistable family over~$R$ and $\widetilde D$ a rank~2 divisor
on the general fiber of~$\X$ with $\rho(\widetilde D) = D_M$. First, we make the
simplifying assumption that the components of the special fiber are
geometrically irreducible, so that we can compute specializations in~$\X$,
without needing to take further field extensions. Let $\overline D$ denote the
closure in~$\X$ of~$\widetilde D$. By assumption, $H^0(\X, \cO(\overline D))$ is
isomorphic to the free $R$-module~$R^3$. By restricting a basis of these
sections to the special fiber~$\X_0$, we have a rank~$2$ linear series on the
reducible curve~$\X_0$.

If $\overline D$ intersected a node of $\X_0$, then it would intersect both
components of~$\X_0$ containing that node, so $\rho(D)$ would have positive
multiplicity on two adjacent vertices. However, $\Gamma_M$ is bipartite and the
divisor $D_M$ is supported on one of these parts, corresponding to the elements
of the matroid, so $\overline D$ cannot intersect any of the nodes
of~$\X_0$. Thus, the base locus of our linear series consists of a finite number
of smooth points of~$\X_0$. Since the base locus consists of smooth points, we
can subtract the base points to get a regular, non-degenerate morphism
$\phi\colon \X_0 \rightarrow \PP_k^2$.

By the assumption that $\widetilde D$ specializes to $D_M$, we have an upper
bound on the degree of $\phi$ restricted to each component of~$\X_0$.
For a flat~$f$ of~$M$, the corresponding component~$C_f$ has degree~$0$
under~$\phi$, so $\phi(C_f)$ consists of a single point.
For an element~$e$, the corresponding component $C_e$ has either
degree~1 or~0 depending on whether the intersection of~$\overline
D$ with~$C_e$ is contained in the base locus. If the intersection is in the base locus, then $C_e$ again maps to
a point, and if not, $C_e$ maps isomorphically to a line in~$\PP_k^2$.
Thus, the image $\phi(\X_0)$ is a union of lines in~$\PP_k^2$, which we will
show to be a dual realization of the matroid~$M$.
Let $f$ be a flat of~$M$. Since the
component of $\X_0$ corresponding to~$f$ maps to a point, the images of the
components corresponding to the elements in~$f$ all have a common point of
intersection.

Now let $e_1$ be an element of~$M$ and suppose that the component $C_{e_1}$ maps
to a point $\phi(C_{e_1})$. Since every other element $e'$ is in a flat with
$e_1$, that means that $\phi(C_{e'})$, the image of the corresponding component
must contain the point~$\phi(C_{e_1})$. Since $\phi$ is non-degenerate, there
must be at least one component $C_{e_2}$ which maps to a line. Let $e_3$ be an
element of~$M$ which completes $\{e_1, e_2\}$ to a basis. Thus, the flat
containing $e_2$ and $e_1$ is distinct from the flat containing $e_2$ and~$e_3$.
Since $\phi$ maps $C_{e_2}$ isomorphically onto its image, this means that
$\phi(C_{e_3})$ must meet $\phi(C_{e_2})$ at a point distinct from the point
$\phi(C_{e_1})$. Thus, $\phi(C_{e_3})$ must be equal to $\phi(C_{e_2})$. Any
other element~$e''$ in~$M$ forms a basis with $e_1$ and either $e_2$ or~$e_3$
(or both). In either case, the same argument again shows that $C_{e''}$ must map
to the same line as $C_{e_2}$ and~$C_{e_3}$. Thus, this line would be the entire
image of $\phi$, which again contradicts the assume non-degeneracy. 
Therefore, we conclude that $\phi$ maps each
component $C_e$ corresponding to an element~$e$ isomorphically onto a line
in~$\PP_k^2$. We've already shown that for any set of elements in a flat, the
corresponding lines intersect at the same point. Moreover, because each
component~$C_e$ maps isomorphically onto its image, distinct flats must
correspond to distinct points in~$\PP_k^2$. Thus, $\phi(\X_0)$ is a dual
realization of the matroid~$M$.

If the components of the special fiber are not geometrically irreducible, then
we can find a finite \'etale extension~$R'$ of~$R$ over which they are.
In our construction of a realization of~$M$ over the residue field of~$R'$, we
can assume that we have chosen a basis of $H^0(\X \times_R R', \cO)$ that is
defined over~$R$. Then, the matroid realization will be the base extension of a
map of $k$-schemes $\X_0 \rightarrow \PP^2_k$. We let $k'$ be the Galois closure
of the residue field of $R'$. Then $\Gal(k'/k)$ acts on the realization of~$M$
over $k'$, but the total collection of lines is defined over~$k$, and thus
invariant. Thus, $M$ has a Galois-invariant realization over an extension of~$k$
as desired.
\end{proof}

For the converse of Theorem~\ref{thm:matroid-lifting}, we need to consider
realizations of matroids over discrete valuations ring~$R$, by which we mean
$R$-points in $\PP^2$ whose images in both the residue field and the fraction
field realize~$M$. For example, if $R$ contains a field over which $M$ is
realizable, then $M$ is realizable over~$R$. We say that $M$ has a
\defi{Galois-invariant realization over an extension of~$R$} if there exists a
finite, flat scheme in $\PP^2_R$ whose special and general fiber are
Galois-invariant realizations of~$M$ over extensions of the residue field and
fraction field of~$R$, respectively.

In the following theorem,
a \defi{complete flag} refers to the pair of an
element~$e$ and a flat~$f$ such that $e$ is contained in~$f$.

\begin{thm}\label{thm:construction}
Let $R$ be a discrete valuation ring with residue field~$k$. Let $M$ be a
simple rank~3 matroid
with a Galois-invariant realization over an extension of~$R$.
Assume that $\lvert k \rvert > m - 2n + 1$, where $n$ is the number of elements
of~$M$ and $m$ is the number of complete flags.
Then $\Gamma_M$ and $D_M$ lift over~$R$.
\end{thm}

Note that Theorem~\ref{thm:construction} does not make any completeness or other
assumptions on the DVR beyond the cardinality of the residue field. In contrast,
ignoring $D_M$ and its rank, a semistable model $\mathcal X$ is only known to
exist for an arbitrary graph when the valuation ring is
complete~\cite[Thm.~B.2]{baker}.

We construct the semistable family in Theorem~\ref{thm:construction} using a
blow-up of projective space. We begin with a computation of the Euler
characteristic for this blow-up.

\begin{lem}\label{lem:euler}
Let $S$ be the blow-up of $\PP^2_K$ at the points of intersection of an
arrangement of $n$ lines. If $A$ is the union of the strict transforms of the
lines and the exceptional divisors, then the dimension of $H^0(S, \cO(A))$ is at
least $2n+1$.
\end{lem}

\begin{proof}
We first use Riemann-Roch to compute that $\chi(\cO(A))$ is $2n+1$. Let $m$ be the
number of complete flags of~$M$, the matroid of the line arrangement~$A$.
We let $H$ denote
the pullback of the class of a line on $\PP^2$ and $C_f$ to denote the
exceptional lines. Then, we have the following linear equivalences
\begin{align*}
A &\sim nH - \sum_{f} (\lvert f \rvert - 1) C_f \\
K_S &\sim -3H + \sum_f C_f
\end{align*}
Now, Riemann-Roch for surfaces tells us that
\begin{align}
\chi(\cO(A)) = \frac{A^2 - A \cdot K_S}{2} + 1
&= \frac{n^2 - \sum_f (\lvert f \rvert - 1)^2  + 3n
- \sum_f (\lvert f \rvert -1)}{2} + 1 \notag \\
&= \frac{n^2 + 3n - \sum_{f} \lvert f \rvert (\lvert f \rvert - 1)}{2} + 1.
\label{eq:riemann-roch}
\end{align}
We can think of the summation $\sum_f \lvert f \rvert(\lvert f \rvert - 1)$ as
an enumeration of all triples of a flat and two distinct elements of the flat.
Since two distinct elements uniquely determine a flat, we have the identity that
$\sum_f \lvert f\rvert^2 = n(n-1)$, so (\ref{eq:riemann-roch}) simplifies to
$\chi(\cO(A)) = 2n + 1$.

It now suffices to prove that $H^2(S, \cO(A))$ is zero, which is equivalent, by
Serre duality, to showing that
$K_S-A$ is not linearly equivalent to
an effective divisor. The push-forward of $K_S - A$ to $\PP^2$ is $-(n+3)H$,
which is not linearly equivalent to an effective divisor, and thus
$H^2(S, \cO(A))$ must be zero. Therefore,
\begin{equation*}
\chi(\cO(A)) = H^0(S, \mathcal O(A)) - H^1(S, \mathcal
O(A)) \leq H^0(S, \mathcal O(A)),
\end{equation*}
which together with the computation above yields the desired inequality.
\end{proof}

\begin{proof}[Proof of Theorem~\ref{thm:construction}]
We first assume that $M$ is realizable over~$R$, and then at the end, we will
handle Galois-invariant realizations over extensions.
Thus, we can fix a dual realization of~$M$ as a set of lines in~$\PP_R^2$, and
let $S$ be
the blow-up of $\PP_R^2$ at all the points of intersections of the lines.
We let the divisor~$A \subset S$ be the sum of the strict transforms of the
lines and the exceptional divisors. Note that $A$ is a simple normal
crossing divisor whose dual complex is~$\Gamma_M$. As in the proof of
Theorem~\ref{thm:matroid-lifting}, we denote the components of~$A$ as~$C_f$
and~$C_e$ corresponding to a flat~$f$ and an element~$e$ of~$M$ respectively.

We claim that $A$ is a base-point-free divisor on~$S$. Any two lines of
the matroid configuration are linearly equivalent in $\PP_R^2$. The preimage of
a linear equivalence between lines corresponding to elements~$e$ and~$e'$
is the divisor:
\begin{equation*}
[C_{e'}] - [C_e] + \sum_{f \colon e' \in f, e \not\in f} [C_f]
- \sum_{f \colon e \in f, e' \not\in f} [C_f].
\end{equation*}
Thus, we have a linear equivalence between $A$ and a divisor which does not
contain~$C_e$, nor $C_f$ for any of the flats containing $e$ but not $e'$. By
varying $e$ and $e'$, we get linearly equivalent divisors whose common
intersection is empty.

We now look for a function~$g \in H^0(S, \cO(A)) \otimes_R k$ which does not
vanish at the nodes of~$A$. For each of the $m$~nodes, the condition of
vanishing at that node amounts to one linear condition on $H^0(S, \cO(A))
\otimes_R k$. Since $A$ is base-point-free, this is a non-trivial linear
condition, defining a hyperplane. Moreover, because of the degrees of the
intersection of~$A$ with its components, the only functions vanishing on all of
the nodes are multiples of the defining equation of~$A$. If the residue field is
sufficiently large, then we can find an element $g \in H^0(S, \cO(A)) \otimes_R
k$ avoiding these hyperplanes, and $\lvert k \rvert > m - 2n + 1$ is sufficient
by Lemmas~\ref{lem:euler} and~\ref{lem:avoidance}. Now we lift $g$ to
$\widetilde g \in H^0(S, \cO(A))$, and set $\X$ to be the scheme defined by $h +
\pi \widetilde g$, where $h$ is the defining equation of~$A$ and $\pi$ is a
uniformizer of~$R$. It is clear that $\X$ is a flat family of curves over $R$
whose special fiber is $A$ and thus has dual graph $\Gamma_M$. It remains to
check that $\X$ is regular and for this it is sufficient to check the nodes of
$A_k$. In the local ring of a node, $h$ is in the square of the maximal ideal,
but by construction $\pi \widetilde g$ is not, and thus, at this point $\X$ is
regular.

Finally, we can take $D$ to be the preimage of any line in $\PP_R^2$ which
misses the points of intersection. Again, by Lemma~\ref{lem:avoidance} below, it
is sufficient that $\lvert k\rvert > \ell - 2$, where $\ell$ is the number
of flats. We claim that $m - 2n + 1 \geq \ell - 2$, and we have assumed that
$\lvert k \rvert > m-2n+1$.
This claimed inequality can be proved using induction similar to the proof of
Theorem~\ref{thm:classification}, but it also follows from Riemann-Roch for
graphs~\cite[Thm.~1.12]{baker-norine-rr}. Since $\Gamma_M$ has genus $m - \ell -
n + 1$, then the Riemann-Roch inequality tells us that
\begin{equation*}
2 = r(D_M) \geq n - (m - \ell - n + 1) = \ell - m + 2n - 1,
\end{equation*}
which is equivalent to the claimed inequality.

Now, we assume that $M$ may only have a Galois-invariant realization over an
extension of~$R$. We can construct the blow-up~$S$ in the same way, since the
singular locus of the line configuration is defined over~$R$. Again, the
divisor~$A$ is base-point-free, because we have already checked that it is base
point free after passing to an extension where the lines are defined.
Finally, we need to choose the function~$g$ and the line which pulls back
to~$D$ by avoiding certain linear conditions defined over an extension of~$k$.
However, when restricted to~$k$, these remain linear conditions, possibly of
higher codimension, so we can again avoid them under our hypothesis on~$\lvert
k\rvert$.
\end{proof}

\begin{lem}\label{lem:avoidance}
Let $H_1, \ldots, H_m$ be hyperplanes in the vector space~$k^N$, where $k$ is a
field. Let $c$ denote the codimension of the intersection $H_1 \cap \cdots \cap
H_m$. If $k$ is infinite or if $k$ is the finite field with $q$ elements and $q
> m - c + 1$, then there exists a point in $k^N$ not contained in any
hyperplane.
\end{lem}

\begin{proof}
If $k$ is infinite, the statement is clear, so we assume that $k$ is finite with
$q$ elements.
We first quotient out by the intersection $H_1 \cap \cdots \cap H_m$, so we are
working in a vector space of dimension $c$ and we know that no non-zero vector
is contained in all hyperplanes. This means that the vectors defining the
hyperplanes span the dual vector space, so we can choose a subset as a basis.
Thus, we assume that the first $c$ hyperplanes are the coordinate hyperplanes.
The complement of these consists of all vectors with non-zero coordinates, of
which there are $(q-1)^c$. Each of the remaining $m-c$ hyperplanes contains at
most $(q-1)^{c-1}$ of these. Our assumption is that $q-1 > m-c$, so there must
be at least one point not contained in any of the hyperplanes.
\end{proof}

We illustrate Theorems~\ref{thm:matroid-lifting} and~\ref{thm:construction} and
highlight the difference between their conditions with the following two
examples.

\begin{ex}
Let $M$ be the Fano matroid, which whose realization in $\mathbb P^2_{\mathbb
F_2}$ consists of all $7$ $\mathbb F_2$-points.
Then $M$ is realizable
over a field if and only if the field has equicharacteristic~$2$.
Thus, by Theorem~\ref{thm:matroid-lifting}, a necessary condition for $\Gamma_M$
and $D_M$ to lift over a valuation ring~$R$ is that the residue field of~$R$ has
characteristic~$2$. On the other hand, $M$ has $7$ elements and $21$
complete flags, so Theorem~\ref{thm:construction} says that if $R$ has
equicharacteristic~$2$ and the residue field of~$R$ has more than $8$ elements,
then $\Gamma_M$ and~$D_M$ lift over~$R$. We do not know if there exists a lift
of $\Gamma_M$ and~$D_M$ over any valuation ring of mixed characteristic~$2$.
\end{ex}

\begin{ex}
One the other hand, let $M$ be the non-Fano matroid, which is realizable over~$k$ if and
only $k$ has characteristic not equal to~$2$. Moreover, $M$ is realizable over
any valuation ring~$R$ in which $2$ is invertible. Thus, $\Gamma_M$ and $D_M$
lift over a valuation ring $R$, only if the residue field of~$R$ has
characteristic different than~$2$ by Theorem~\ref{thm:matroid-lifting}. The
converse is true, so long as the residue field has more than~$11$ elements by
Theorem~\ref{thm:construction}.
\end{ex}

Since
Theorems~\ref{thm:matroid-lifting} and~\ref{thm:construction} refer to
Galois-invariant realizations, we will need the following lemma to relate such
realizations with ordinary matroid realizations.

\begin{lem}\label{lem:break-symmetry}
Let $M$ be a matroid of rank 3. Then there exists a matroid~$M'$ such that for any
infinite field~$k$, the following are equivalent:
\begin{enumerate}
\item $M$ has a realization over~$k$.
\item $M'$ has a realization over~$k$.
\item $M'$ has a Galois-invariant realization over an extension of~$k$.
\end{enumerate}
\end{lem}

\begin{proof}
We use the following construction of an extension of a matroid. Suppose that $M$
is a rank~$3$ matroid and $f$ is a flat of~$M$. We construct a matroid~$M''$
which contains the elements of~$M$, together with an additional element~$x$. The
flats of $M''$ are those of~$M$, except that $f$ is replaced by $f \cup \{x\}$,
and two-element flats for $x$ and every element not in~$f$. By
repeating this construction, we can construct a matroid~$M'$ such
every flat which comes from one of the flats of~$M$ has a different number of
elements.

Now we prove that the conditions in the lemma statement are equivalent for this
choice of~$M'$. First, assume that $M$ has a realization over an infinite
field~$k$. We can inductively extend this to a realization of~$M'$. At each
step, when adding an element~$x$ as above, it is sufficient to place $x$ at a
point along the line corresponding to~$f$ such that it does not coincide with
any
of the other points, and it is not contained in any of the lines spanned by two
points not in~$f$. We can choose such a point for~$x$ since $k$ is infinite.
Second, if $M'$ has a realization over $k$, then by definition, it has a
Galois-invariant realization over an extension of~$k$.

Finally, we suppose that $M'$ has a Galois-invariant realization over an
extension of~$k$ and we want to show that $M$ has a realization over~$k$.
Suppose we have a realization over a Galois extension~$k'$ of~$k$. Since all the
flats from the original matroid contain different numbers of points, the Galois
group does not permute the corresponding lines in the realization. Therefore,
the lines and thus also the points from the original matroid~$M$ must be defined
over $k$. Therefore, the restriction of this realization gives a realization
of~$M$ over $k$, which completes the proof of the lemma.
\end{proof}

\begin{proof}[Proof of Theorem~\ref{thm:universality}]
As in the statement of the theorem, let $X$ be a scheme of finite type
over~$\ZZ$. We choose an affine open cover of~$X$ and let $\widetilde X$ be the
disjoint union of these affine schemes. By the scheme-theoretic version of
Mn\"ev's universality theorem, either Theorem~1.14 in \cite{lafforgue} or our
Theorem~\ref{thm:mnev}, there is a matroid~$M$ of rank~3 whose realization space
is isomorphic to an open subset~$U$ of~$\widetilde X \times \mathbb A^N$ and $U$
maps surjectively onto~$X$. Now let $M'$ be the matroid as in
Lemma~\ref{lem:break-symmetry} and we claim that $\Gamma_{M'}$ and~$D_{M'}$ have
the desired properties for the theorem.

Let $k$ be any infinite field, and then $X$ clearly has a $k$-point if and only
if $\widetilde X$ has a $k$-point. Likewise, since $k$ is infinite, any
non-empty subset of $\mathbb A^N_k$ has a $k$-point, so $U$ also has a $k$-point if
and only if $X$ has a $k$-point. By Lemma~\ref{lem:break-symmetry}, these
conditions are equivalent to $M'$ having a Galois-invariant realization over
an extension of~$k$. Supposing that $X$ has a $k$-point and thus $M'$ has a
realization over $k$, then $\Gamma_{M'}$ and $D_{M'}$ have a lifting over
$k[[t]]$ by Theorem~\ref{thm:construction}. Conversely, if $D_{M'}$ has a
lifting over $k[[t]]$, then $M'$ has a Galois-invariant realization over an
extension of~$k$ by Theorem~\ref{thm:matroid-lifting}, and thus $M$ has a
realization over~$k$ by Lemma~\ref{lem:break-symmetry}, so $X$ has a $k$-point.
\end{proof}

\section{Brill-Noether theory}\label{sec:brill-noether}

In this section, we take a detour and look at connections to Brill-Noether
theory and the analogy between limit linear series and tropical divisors.
In the theory of limit linear series, a key technique is the observation
that if the moduli space of limit linear series on the degenerate curve has the
expected dimension then it lifts to a linear
series~\cite[Thm.~3.4]{eisenbud-harris}. Here, the expected dimension of limit
linear series of degree~$d$ and rank~$r$ on a curve of genus~$g$ is $\rho(g, r,
d) = g - (r+1)(g + r - d)$. It is natural to ask if a tropical analogue of this
result is true: if the dimension of the moduli space of divisor classes on a
tropical curve of degree~$d$ and rank at least $r$ has (local) dimension
$\rho(g,r,d)$, then does every such divisor lift?
See~\cite{cjp} for further discussion and one case with an affirmative answer.
The main result of this section is that
the matroid divisors and graphs constructed in this paper do not
provide a negative answer to the above question.

We begin with the following
classification:

\begin{figure}
\includegraphics{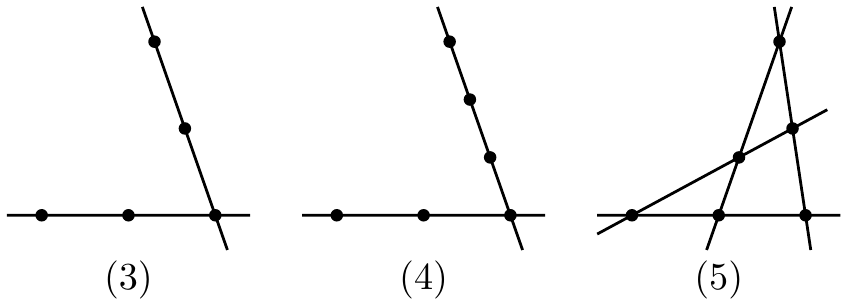}
\caption{Matroids from Theorem~\ref{thm:classification}}
\label{fig:matroids}
\end{figure}

\begin{thm}\label{thm:classification}
Let $M$ be a rank~$3$ simple matroid, with $g$ and~$d$ equal to the genus
of~$\Gamma_M$ and degree of~$D_M$ respectively. If $\rho = \rho(g, 2, d) \geq
0$, then $M$ is one of the following matroids:
\begin{enumerate}
\item\label{mat:extension} The one element extension of the uniform matroid
$U_{2,n-1}$, with $\rho = n-2$.
\item\label{mat:uniform} The uniform matroid $U_{3,4}$, with $\rho = 0$.
\item\label{mat:base} The matroid defined by the vectors: $(1,0,0)$, $(1,0,1)$,
$(0,0,1)$, $(0,1,1)$, $(0,1,0)$, with $\rho = 1$.
\item\label{mat:two-lines} The matroid in the previous example together with
$(1, 0, \lambda)$ for any element $\lambda$ of the field other than $1$ and~$0$,
with $\rho = 0$.
\item\label{mat:four-lines} The matroid consisting of the point of intersection
between any pair in a collection of $4$ generic lines, for which we can take the
coordinates to be the vectors from (3) together with $(1,1,1)$, with $\rho = 0$.
\end{enumerate}
\end{thm}

The last three cases of Theorem~\ref{thm:classification} are illustrated in
Figure~\ref{fig:matroids}.

\begin{proof}
We first compute the invariants for the graph~$\Gamma_D$ and divisor~$D_M$
constructed in Section~\ref{sec:divisors}. As before, we let $n$ be the
number of elements of~$M$, $\ell$ the number of flats, and $m$ the number of
complete flags. Since $\Gamma_D$ consists of $m$ edges and $n+\ell$ vertices, it
has genus $m - n - \ell + 1$. It is also immediate from its definition that
$D_M$ has degree~$n$. Thus, the expected dimension of rank~$2$ divisors is
\begin{equation}\label{eq:rho}
\rho = 
m - n - \ell + 1 - 3((m - n - \ell + 1) + 2 - n)
= 5n + 2 \ell - 2 m - 8
\end{equation}

Now, assume that $\rho$ is non-negative for~$M$ and we consider what what
happens to~$\rho$ when we remove a single element~$e$
from a matroid, where
$e$ is not contained in all bases. For every flat containing $e$, we decrease
the number of complete flags by~$1$ if that flat contains at least $3$ elements,
and if it contains $2$ elements, then we decrease the number of flags by $2$ and
the number of flats by~$1$. Thus, by~(\ref{eq:rho}), $\rho$ drops by $5-2s$,
where $s$ is the number of flats in~$M$ which contain~$e$. Since $e$ must be
contained in at least $2$~flats, either $M \setminus e$, the matroid formed by
removing $e$ has positive $\rho$ or $e$ is contained in exactly $2$ flats.

We first consider the latter case, in which $e$ is contained in exactly two
flats, which we assume to have cardinality $a+1$ and $b+1$ respectively. The
integers $a$ and~$b$ completely determine the matroid because all the other
flats consist of a pair of elements, one from each of these sets. Thus, there
are $ab+2$ flats and $2ab + a + b + 2$ complete flags. By using (\ref{eq:rho}),
we get
\begin{equation*}
\rho = 5(a + b + 1) + 2(ab + 2) - 2(2ab + a + b + 2) - 8
= -2ab + 3a + 3b -3.
\end{equation*}
One can check that, up to swapping $a$ and~$b$, the only non-negative
values of this expression are when $a = 1$ and $b$ is arbitrary or $a=2$ and $b$
is $2$ or~$3$. These correspond to cases (\ref{mat:extension}),
(\ref{mat:base}), and (\ref{mat:two-lines}) respectively from the theorem
statement.

Now we consider the case that $e$ contained in more than two flats, in which
case $M \setminus e$ satisfies $\rho > 0$. By induction on the number of
elements, we can assume that $M \setminus e$ is on our list, in which case the
possibilities with $\rho > 0$ are (\ref{mat:extension}) and
case~(\ref{mat:base}). For the former matroid, if $e$ is contained in a flat of
$M \setminus e$, then $M$ is a matroid of the type from the previous paragraph,
with $a$ equal to $1$ or~$2$. On the other hand, if $e$ contained only in
$2$-element flats, then $e$ is contained in $n-1$ flats, so
\begin{equation*}
\rho(M) = \rho(M \setminus e) + 5 - 2(n-1) = (n-3) + 7 - 2n = 4 - n.
\end{equation*}
The only possibility is $n=4$, for which we get (\ref{mat:uniform}), the
uniform matroid. Finally, if $M \setminus e$ is the matroid in
case~(\ref{mat:base}), then the only relevant possibilities are those for which
$e$ is contained in at most $3$ flats, for which the possible matroids are
(\ref{mat:two-lines}) or~(\ref{mat:four-lines}).
\end{proof}

\begin{prop}
If $R$ is a DVR and $M$ is one of the matroids in
Theorem~\ref{thm:classification}, then $M$ has a Galois-invariant realization
over an extension of~$R$.
\end{prop}

\begin{proof}
The matroids~(\ref{mat:uniform}), (\ref{mat:base}) and~(\ref{mat:four-lines})
are regular matroids, i.e.\ realizable over~$\ZZ$, so they are \emph{a fortiori}
realizable over any DVR. Moreover, the other matroids in
case~(\ref{mat:extension}) and~(\ref{mat:two-lines}) are realizable over $R$ so
long as the residue field has at least $n-2$ and $3$ elements respectively. We
will show that if the residue field is finite, then the one-element extension of
$U_{2,n-1}$ has a Galois-invariant realization over~$R$. The other case is
similar.

Let $M$ be the one-element extension of~$U_{2,n-1}$ and suppose the residue
field~$k$ is finite. We choose a polynomial with coefficients of
degree~$n-1$ in~$R$ whose reduction to~$k$ is square-free. Adjoining the
roots of this polynomial defines an unramified extension~$R'$ of~$R$, and we
write $a_1, \ldots, a_{n-1}$ for its roots in~$R$. Then, the vectors $(1,a_1,0),
\ldots, (1,a_{n-1}, 0), (0, 0,1)$ give a Galois-invariant realization of~$M$
over~$R'$, which is what we wanted to show.
\end{proof}

\section{Quantitative Mn\"ev universality}\label{sec:mnev}

In this section, we prove a quantitative version of Mn\"ev universality over
$\Spec \ZZ$ with Theorem~\ref{thm:quantitative-characteristic} as our desired
application. We follow the strategy of \cite[Thm.~1.14]{lafforgue}, but use the
more efficient building blocks used in, for example,~\cite{lee-vakil}. We pay close attention
to the number of points used in our construction in order to get effective
bounds on the degree of the corresponding matroid divisor. These bounds are
expressed in terms of the following representation.

\begin{defn}\label{def:elementary-monic}
Let $S_n$ denote the polynomial ring $\ZZ[y_1,
\ldots, y_n]$. In the extension $S_n[t]$, we also introduce the coordinates
$x_i$ defined by $x_0 = t$ and $x_i = y_i + t$ for $1 \leq i \leq n$. In
addition, for $n < i \leq m$, suppose we have elements $x_i \in S_n[t]$ such
that:
\begin{enumerate}
\item Each~$x_i$ is defined as one of $x_i = x_j + x_k$, $x_i = x_j
x_k$, or $x_i = x_j + 1$, where $j, k < i$.
\item Each $x_i$ is monic as a polynomial in~$t$ with coefficients in~$S_n$.
\end{enumerate}
The coordinates $x_i$ for $1 \leq i \leq n$ will be called \defi{free variables}
and the three operations for defining new variables in~(1) will be called
\defi{addition}, \defi{multiplication}, and \defi{incrementing}, respectively.

Moreover, we suppose we have finite sets of equalities~$E$ and inequalities~$I$
consisting of
pairs $(i,j)$ such that $x_i - x_j$ is in $S_n \subset S_n[t]$. 
We then say that the algebra:
\begin{equation*}
S_n[(x_{i'}-x_{j'})^{-1}]_{(i',j')\in I}/
\langle x_i - x_j \mid (i, j) \in E\rangle
\end{equation*}
has an \defi{elementary monic representation} consisting of the above
data, namely, the
integers~$n$ and~$m$, the expression of each $x_i$ as an addition,
multiplication, or increment for $n < i \leq m$, and the sets of equalities and
inequalities.
\end{defn}

The inequalities $I$ in Definition~\ref{def:elementary-monic} are not strictly
necessary because an inverse to $x_i-x_j$ can always be introduced as a
new variable, but the direct use of inequalities may be more efficient, such as
in the proof of Theorem~\ref{thm:quantitative-characteristic}.

\begin{prop}\label{prop:monic-representation}
There exists an elementary monic representation of any finitely generated
$\ZZ$-algebra.
\end{prop}

\begin{proof}
We begin by presenting the $\ZZ$-algebra~$R$ as
\begin{equation*}
R = \ZZ[y_1, \ldots, y_n]/\langle f_1 - g_1, \ldots, f_m - g_m \rangle,
\end{equation*}
where each polynomial~$f_k$ and~$g_k$ has positive integral coefficients.
Then $f_k$ and~$g_k$ can be constructed by a sequence of
multiplication and addition operations applied to the variables~$y_i$ and the
constant~$1$. Obviously, we can assume that our multiplication never involves
the constant~$1$. To get an elementary monic representation, we first replace
the variables $y_i$ with $x_i = y_i + t$ in the constructions of $f_k$
and~$g_k$, with some adjustments, as follows. Since Definition~\ref{def:elementary-monic} does not allow addition
of $1$ with itself, we replace such operations by first introducing a new
variable $x_i = x_0 + 1 = t + 1$ and then adding $1$ to $x_i$. Similarly, in
order to satisfy the second condition of Definition~\ref{def:elementary-monic},
when adding two variables $x_i$ and~$x_j$ which are both monic of the same
degree $d$ in $t$, we first compute an intermediate $x_{i'} = t^{d+1} + x_i$ and
then the sum $x_{i'} + x_j = t^{d+1} + x_i + x_j$. In this way, we ensure that
all of the $x_i$ variables are monic in~$t$. Moreover, $x_i$ and~$x_j$ agree
with $f_k$ and~$g_k$, respectively, modulo~$t$, but in order to be able to
have an equality $x_i = x_j$ in an elementary monic representation, the
difference $x_i - x_j$ have to not involve $t$.

Therefore, we want to replace $x_i$ and $x_j$ by polynomials $x_i'$ and $x_j'$
which agree with $x_i$ and $x_j$ modulo $t$, but such that the difference $x_i'
- x_j'$ does not involve $t$. We do this by double induction, first,
on the maximum total degree in the $y$ variables of the terms of $x_i - x_j$
that involve $t$, and second, on the number of terms of that degree.

Thus, we suppose that $c t^s y_1^{a_1} \cdots y_n^{a_n}$ is a term of $x_i -x_j$
whose total degree in the $y$ variables is maximal among terms with $s > 0$.
By swapping $i$ and~$j$ if necessary, we can assume that $c$ is
positive. We then use multiplication operations to construct:
\begin{equation*}
x_\ell = t^s x_1^{a_1} \cdots x_n^{a_n}
= t^s (t + y_1)^{a_1} \cdots (t + y_n)^{a_{n}}
= t^s y_1^{a_1} \cdots y_m^{a_m} + \ldots,
\end{equation*}
where the final ellipsis denotes omitted terms with lower degree in the $y$
variables. First suppose that $x_j$ and $x_\ell$ have different degrees
in the $t$
variable. Then, we can use $c$ addition operations to construct $x_{j'} = x_j +
c x_\ell$, and we set $i' = i$. On the other hand, if $x_j$ and $x_\ell$
have the same degree in~$t$, then let $d$ be an integer larger than the
$t$-degree of $x_i$, $x_j$, and $x_\ell$, and we use additions to construct
$x_{j'} = x_j + c x_\ell + t^d$ and $x_{i'} = x_i + t^d$. In either case,
$x_{i'}$ and $x_{j'}$ equal $x_i$ and $x_j$, respectively, modulo~$t$. Moreover,
the term $c t^s y_1^{a_1} \cdots y_m^{a_m}$ has been eliminated from $x_{i'} -
x_{j'}$, while only
introducing new terms which have lower degree in the $y$ variables. Thus,
by induction, we can eliminate all terms of $x_i - x_j$ which involve $t$ and
have maximal total degree in the $y$ variables among such terms, and by the
second level of induction, we can eliminate such terms in all degrees. We
therefore have constructions of variables $x_{i''}$ and $y_{j''}$ such that
$x_{i''} - x_{j''} = f_k - g_k$. Setting these equal for $k
= 1, \ldots, m$ gives us the elementary monic representation of $R$.
\end{proof}

Given a matroid~$M$, its possible realizations form a scheme, called the
realization space of the matroid~\cite[Sec.~9.5]{katz}. Explicitly, given a
rank~3
matroid with $n$ elements, each flat of the matroid defines a closed,
determinantal condition in $(\PP^2_\ZZ)^n$ and each triple of elements which is
not in any flat defines an open condition by not being collinear. The
\defi{realization space} is the quotient by $\PGL_3(\ZZ)$ of the
scheme-theoretic intersection of these conditions. We will only consider the
case when this action is free, in which case the quotient will be an affine
scheme over $\Spec \ZZ$.

\begin{thm}[Mn\"ev universality]\label{thm:mnev}
For any finite-type $\ZZ$-algebra~$R$, there exists a rank~$3$ matroid~$M$ whose
realization space is an open subset $U \subset \mathbb A^N \times \Spec R$ such
that $U$ projects surjectively onto $\Spec R$.

Moreover, if $R$ has an elementary monic representation with $n$~free variables,
$a$~additions, $m$~multiplications, $o$ increments, $e$~equalities, and
$i$~inequalities, then $M$ has
\begin{equation*}
3n + 7a + 7o + 6m + 5e + 6i + 6
\end{equation*}
elements, and
\begin{equation*}
N = 3(n + a + o + m + e + i) + 1.
\end{equation*}
\end{thm}

\begin{proof}
By Proposition~\ref{prop:monic-representation}, we can assume that $R$ has an
elementary monic representation. Both the matroid and its potential
realization will be built up from the elementary monic
representation, beginning with the free variables and then applying the
addition, multiplication, and increment operations. We describe the
constructions of both the matroid and the realization in parallel for ease of
explaining their relationship.

We begin with the free variables. For $x_0 = t$ and for each free variable~$x_i$
of the representation, we have a line, realized generically, passing through a
common fixed point. In the figures below, we will draw these horizontally so
that the common point is at infinity. On each of these lines we have 3
additional points, whose positions along the line are generic. Our convention
will always be that points whose relative position is not specified are generic.
In other words, unless otherwise specified to lie on a line, each pair of points
correspond to a 2-element flat.

From each set of $4$~points on one of these free variable lines, we can take the
cross-ratio, which is invariant under the action of $\PGL_3(\ZZ)$. Therefore, by
taking the cross-ratio on each line as the value for the corresponding
coordinate~$x_i = t+ y_i$ or $x_0 = t$, we define a morphism from the
realization space of the matroid defined thus far to $\mathbb A^{n+1} = \Spec
S_n[t]$, where $S_n = \ZZ[y_1, \ldots, y_n]$ as
in Definition~\ref{def:elementary-monic}. Our goal with the remainder of the
construction is to constrain the realization such that the projection to $\Spec
S_n$ is surjective onto $\Spec R \subset \Spec S_n$.

Concretely, the cross-ratio is the position of one point on
the line in coordinates where the other points are at $0$, $1$, and~$\infty$.
For us, the point common to all variable lines will be at~$\infty$, so we will
refer to the other points along the line as the ``$0$'' point, the ``$1$''
point, and the variable point. We will next embed the operations of addition,
multiplication, and incrementing from the elementary monic representation. The
result of each of these operations will be encoded as the cross-ratio of $4$ points on a
generic horizontal line, in the same way as with the free variables.

\begin{figure}
\includegraphics{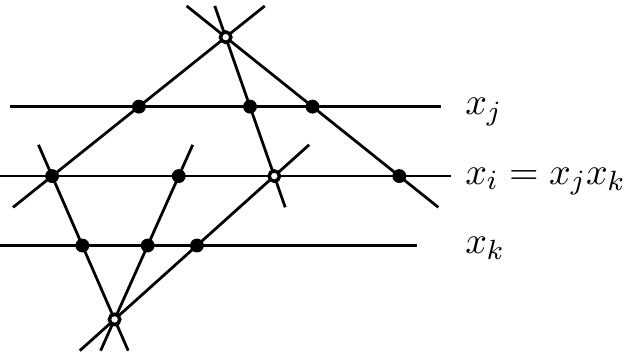}
\caption{Configuration for computing multiplication. The cross-ratios of the
solid circles, together with the horizontal point at infinity define the
variables. On each line, the solid circles are, from left to right, the ``$0$''
point, the variable point, and the ``$1$'' point of the cross-ratio defining the
variable. The empty circles are auxiliary points. }
\label{fig:multiplication}
\end{figure}
First, multiplication of distinct variables $x_i = x_j x_k$ is constructed as in
Figure~\ref{fig:multiplication}, where the $x_j$ and $x_k$ lines refer to the
lines previously constructed for those variables and the other points are new.
We can choose the horizontal line for $x_i$ as well as the additional points
generically so that there none are collinear with previously constructed points.
Then, one can check that the cross-ratio of the solid points on the central line
is the product of the cross-ratios on the other two lines. Set-theoretically,
this claim follows from the fact that projections between parallel lines
preserve ratios of distances, and so measuring from the leftmost point of the
$x_i$ line, the top projection ensures that the ratio between the distances to
the empty circle and the rightmost circle is $x_j$. Likewise, the lower
projection ensures that the ratio of the distances to the center solid circle
and the empty circle is $x_k$ and so $x_i = x_jx_k$ appears as the product of
the ratios. If $j$ equals~$k$,
the diagram may be altered by moving the corresponding lines so that they
coincide. In either case, the construction uses $6$ additional points.

\begin{figure}
\includegraphics{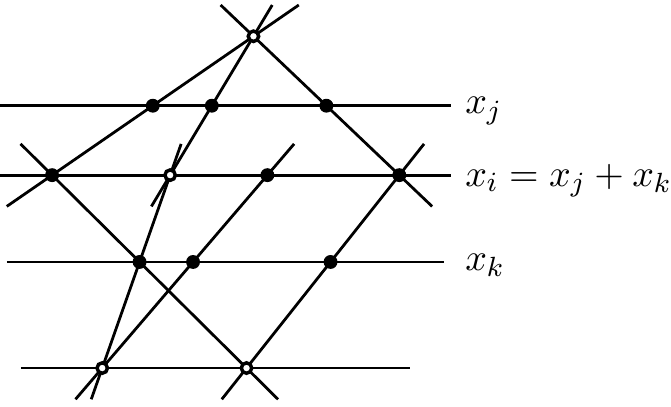}
\caption{Configuration for computing addition. The solid and empty circles
represent the variables and auxiliary points respectively, as in
Figure~\ref{fig:multiplication}.}
\label{fig:addition}
\end{figure}
Second, the addition of variables $x_i = x_j + x_k$ can be constructed as in
Figure~\ref{fig:addition}. As in the case of multiplication, the motivation for
this construction can be understood from the fact that projections scale
distances. In particular, the top projection means that the empty circle and the
outer points on the $x_i$ line encode the variable~$x_j$, and since the two
lower projections are both from points on the same horizontal line, the ratio of
the distance between the inner points on the $x_i$ line to the distance between
the outer points equals the value of~$x_k$. Therefore, the intervals on either
side of the empty circle encode the values of~$x_j$ and $x_k$ and their
concatenation computes $x_i = x_j + x_k$.

In the configuration from Figure~\ref{fig:addition},
there will be an additional coincidence if $x_j = x_k$ in
that the empty circle on the $x_i$ line, the middle point on the
$x_k$ line and a point on the bottom line will be collinear. However,
since $x_j + x_k$,
$x_j$, and~$x_k$ are all monic polynomials in the variable $t$, $x_j - x_k$ is
also monic in~$t$ and so a sufficiently generic choice of $t$ will ensure that
$x_j - x_k$ is non-zero.
Similarly if $x_j = -1$, then the empty point on the $x_i$ line, the rightmost
point on the $x_k$ line and a marked point on the bottom line will be collinear,
but this can again be avoided by adjusting~$t$.
For the addition operation,
we've used $7$ additional points.

\begin{figure}
\includegraphics{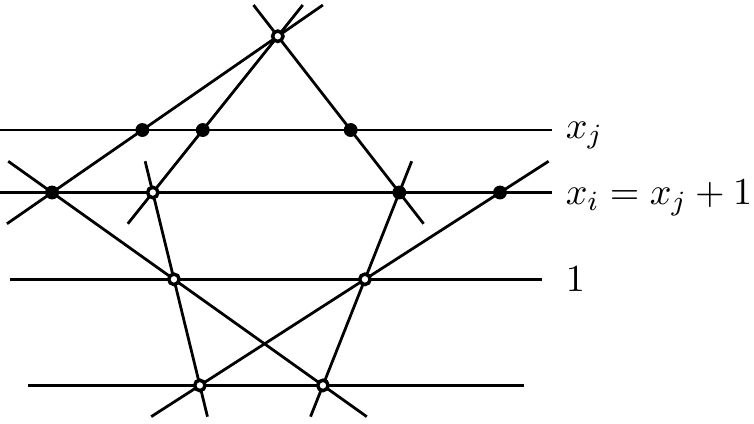}
\caption{Configuration for incrementing. The solid and empty circles represent
the variables and the auxiliary points respectively, as in
Figure~\ref{fig:multiplication}, with the exception that, on the $x_i$ line, the
variable point is the rightmost solid point (since $x_i > 1$).}
\label{fig:increment}
\end{figure}
Third, for incrementing a single variable, $x_i = x_j + 1$, we specialize the
configuration in Figure~\ref{fig:addition} so that $x_k = 1$, giving
Figure~\ref{fig:increment}. The line labeled with~$1$ can be chosen once and
used in common for all increment operations, since it functions as a
representative of the constant~$1$. We've used 7 additional points for each
increment operation, together with $2$ points common to all such operations.

\begin{figure}
\includegraphics{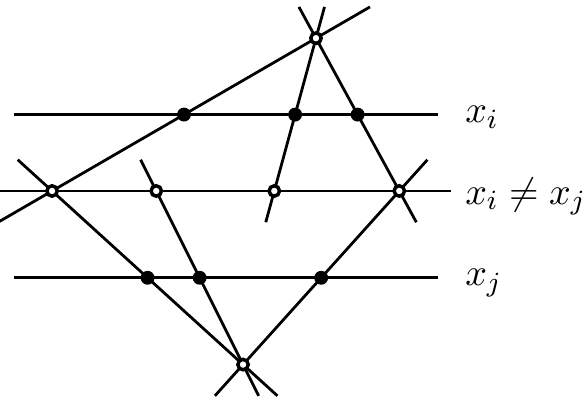}
\caption{Configuration for imposing inequality. The solid and empty circles
represent the variables and auxiliary points respectively, as in
Figure~\ref{fig:multiplication}.}
\label{fig:inequality}
\end{figure}
At this point, the realization space still surjects onto $\Spec S_n$, and so
we still need to impose the equalities and inequalities. Each inequality $x_i
\neq x_j$ can
be imposed using the diagram in Figure~\ref{fig:inequality}, which works by
projecting the two variable points to the same line and getting different
points. By replacing the projections of the two variables to the central line
with the same point, we can use a similar figure to assert equality $x_i = x_j$.
These use $6$ and $5$ additional points respectively.

To summarize, we have agglomerated the configurations in
Figures~\ref{fig:multiplication}, \ref{fig:addition}, \ref{fig:increment},
and~\ref{fig:inequality} to give a matroid whose realization space projects to
$\Spec R \subset \Spec S_n$. The realization of this matroid is determined by
the values of the $y_i$, together with a number of parameters, such as the
height of the horizontal lines, which are allowed to be generic, and thus
the realization space is an open subset of $\mathbb A^N \times \Spec R$.

To show that the projection to $\Spec R$ is surjective, we take any point of
$\Spec R$, which we can assume to be defined over an infinite field. Our
construction of the realization required us to avoid certain coincidences, such
as any $x_i$ being $0$ or~$1$ or an equality $x_j = x_k$ in any addition step.
Each such coincidence only occurs for a finite number of possible values for~$t$
we choose $t$ outside the union of all coincidences, and we can construct a
realization of the matroid.

Finally, we justify the quantitative parts of the theorem statement. The number
of elements of $M$ is computed by summing the number of elements for each of the
building blocks together with $1$ element for the common point on the horizontal
lines, $3$ elements for the variable $x_0 = t$, and $2$ elements for the horizontal line
representing~$1$ in Figure~\ref{fig:increment}.
For the computation of~$N$, we can assume that the coordinates on $\PP^2$
are such that the common point of the horizontal lines is $(1:0:0)$, the points
representing~$1$ are $(0:0:1)$ and $(1:0:1)$, and the ``$0$'' and~``$1$'' points of the $x_0 =
t$ line are $(0:1:0)$ and $(0:1:1)$ respectively. These fix the automorphisms of~$\PP^2$.
Then, one can check that each additional free variable and
each of the building blocks adds $3$~additional generic parameters. Finally, the
value of~$t$ is one more free parameter, which gives the expression for~$N$.
\end{proof}

\begin{figure}
\begin{align*}
x_0 &= t \\
x_1 &= x_0 x_0 = t^2 \\
x_2 &= x_0 x_1 = t^3 \\
x_3 &= x_0 + 1 =  t + 1 \\
&\quad\vdots \\
x_{\ell+2} &= x_{\ell+1} + 1 = t + \ell \\
x_{\ell+3} &= x_1 + x_{\ell+1} = t^2 + t + \ell \\
x_{\ell+4} &= x_{\ell+3} + x_{\ell+1} = t^2 + 2t + 2\ell \\
x_{\ell+5} &= x_0 x_{\ell+4} = t^3 + 2t^2 + 2\ell t \\
x_{\ell+6} &= x_{\ell+2} x_{\ell+2} = t^2 + 2\ell t + \ell^2 \\
x_{\ell+7} &= x_{\ell+6} + 1 = t^2 + 2\ell t + \ell^2 + 1 \\
&\quad\vdots \\
x_{\ell+p-\ell^2 + 6} &= x_{\ell+p-\ell^2+5} + 1 = t^2 + 2\ell t + p \\
x_{\ell+p-\ell^2+7} &= x_2 + x_{\ell+p-\ell^2+6} = t^3 + t^2 + 2\ell t + p \\
x_{\ell+p-\ell^2+8} &= x_1 + x_{\ell+p-\ell^2+7} = t^3 + 2t^2 + 2\ell t + p \\
\end{align*}
\caption{System of equations used in the elementary monic representations of
$\ZZ/p$ and $\ZZ[p^{-1}]$, where $p$ is a prime and
$\ell$ denotes the largest integer less than
$\sqrt{p}$.}\label{fig:elem-monic}
\end{figure}
\begin{proof}[Proof of Theorem~\ref{thm:quantitative-characteristic}]
By Theorems~\ref{thm:matroid-lifting}, \ref{thm:construction},
and~\ref{thm:mnev}, it will be enough to construct sufficiently
parsimonious elementary monic representations of the algebras $\ZZ/p$ and
$\ZZ[p^{-1}]$
and thus matroids $M$ and~$M'$, respectively, representing these
equations. Let $\ell$ be the largest integer less than $\sqrt{p}$. The
elementary monic representations for both $M$ and~$M'$ use the equations shown
in Figure~\ref{fig:elem-monic}.
Then, $\ZZ/p$ can be represented by adding an equality between $x_{\ell+5}$ and
$x_{\ell+p-\ell^2+8}$ and $\ZZ[p^{-1}]$ can be represented by an inequality
between the same pair of variables.

In either case, this representation uses no free variables, $\ell +
p-\ell^2$ increments, $4$ additions, and $4$ multiplications. Thus, by
Theorem~\ref{thm:mnev}, $M$ and~$M'$ have $7(\ell+p-\ell^2) + 64$ and
$7(\ell+p-\ell^2) +
63$ elements respectively. We'll bound the former since it is larger. We first
rewrite the number of elements as
\begin{equation}\label{eq:num-elements}
7(\ell+p-\ell^2) + 64 = p-\ell^2 + 7\ell + 6(p-\ell^2) + 64
\end{equation}

To show that (\ref{eq:num-elements}) is smaller than $p$, we note that since $p
\geq 443$, then $\ell \geq 21$. We now have two cases. First, if $\ell = 21$,
then the largest prime number less than $22^2$ is $479$, so $p-n^2 \leq 38$.
Using this, we can bound (\ref{eq:num-elements}) as 
\begin{equation*}
p - 21^2 + 7\cdot 21 + 6(38) + 64 = p - 2 < p
\end{equation*}
On the other hand, if $\ell \geq 22$, then the choice of $\ell$ means that $p <
(\ell + 1)^2$, so $p - \ell^2 \leq 2 \ell$. Therefore, we can bound
(\ref{eq:num-elements}) as follows:
\begin{align*}
p - \ell^2 + 7\ell + 6(2\ell) + 64
&= p - \ell^2 + 19\ell + 64 \\
&= p - (\ell-19)\ell + 64 \\
&\leq p - 2 < p
\end{align*}
Thus, the number of elements of $M$ and~$M'$ is less than~$p$.

We take the graphs $\Gamma$ and $\Gamma'$ and the divisors $D$ and~$D'$ for the
theorem statement to be the matroid divisors of~$M$ and~$M'$ respectively. Since
$M$ and~$M'$ have fewer than $p$ elements, $D$ and~$D'$ have degree less
than~$p$. Moreover, since $k$ is an infinite field, by
Theorems~\ref{thm:matroid-lifting} and~\ref{thm:construction}, $\Gamma$ and~$D$
lift over $k[[t]]$ if and only if $M$ is representable over~$k$, which means
that the characteristic of~$k$ equals~$p$. Similarly, $\Gamma'$ and $D'$ lift if
and only if the characteristic of $k$ is not~$p$.
\end{proof}

\begin{rmk}
The threshold for~$p$ in Theorem~\ref{thm:quantitative-characteristic} is not
optimal. For example, by using a different construction when $p$ is closer to a
larger square number than to a smaller square, it is
possible to reduce the bound to~$331$.
\end{rmk}

\bibliographystyle{alpha}
\bibliography{lifting}

\begin{thebibliography}{ABBR15b}

\bibitem[AB15]{amini-baker}
Omid Amini and Matthew Baker.
\newblock Linear series on metrized complexes of algebraic curves.
\newblock {\em Math. Ann.}, 362(1):55--106, 2015.

\bibitem[ABBR15a]{abbr1}
Omid Amini, Matthew Baker, Erwan Brugall{\'e}, and Joseph Rabinoff.
\newblock Lifting harmonic morphisms {I}: metrized complexes and {Berkovich}
  skeleta.
\newblock {\em Res. Math. Sci.}, 2(7), 2015.

\bibitem[ABBR15b]{abbr2}
Omid Amini, Matthew Baker, Erwan Brugall{\'e}, and Joseph Rabinoff.
\newblock Lifting harmonic morphisms {II}: Tropical curves and metrized
  complexes.
\newblock {\em Algebra Number Theory}, 9(2):267--315, 2015.

\bibitem[AC13]{amini-caporaso}
Omid Amini and Lucia Caporaso.
\newblock {Riemann-Roch} theory for weighted graphs and tropical curves.
\newblock {\em Adv. Math.}, 240:1--23, 2013.

\bibitem[AK06]{ardila-klivans}
Federico Ardila and Carly Klivans.
\newblock The {Bergman} complex of a matroid and phylogenetic trees.
\newblock {\em J. Combin. Theory Ser. B}, 96(1):38--49, 2006.

\bibitem[Bak08]{baker}
Matthew Baker.
\newblock Specialization of linear systems from curves to graphs.
\newblock {\em Algebra Number Theory}, 2(6):613--653, 2008.
\newblock With an appendix by Brian Conrad.

\bibitem[BN07]{baker-norine-rr}
Matthew Baker and Serguei Norine.
\newblock {Riemann-Roch} and {Abel-Jacobi} theory on a finite graph.
\newblock {\em Adv. Math.}, 215:766--788, 2007.

\bibitem[BN09]{baker-norine}
Matthew Baker and Serguei Norine.
\newblock Harmonic morphisms and hyperelliptic graphs.
\newblock {\em Int. Math. Res. Not.}, 15:2914--2955, 2009.

\bibitem[Cap13]{caporaso}
Lucia Caporaso.
\newblock Rank of divisors on graphs: an algebro-geometric analysis.
\newblock In {\em A celebration of algebraic geometry}, volume~18 of {\em Clay
  Math. Proc.}, pages 45--64. Amer. Math. Soc., Providence, RI, 2013.

\bibitem[CJP15]{cjp}
Dustin Cartwright, David Jensen, and Sam Payne.
\newblock Lifting divisors on a generic chain of loops.
\newblock {\em Canad. Math. Bull.}, 58:250--262, 2015.

\bibitem[Dha90]{dhar}
Deepak Dhar.
\newblock Self-organized critical state of the sandpile automaton models.
\newblock {\em Phys. Rev. Lett.}, 64:1613--1616, 1990.

\bibitem[DSS05]{develin-santos-sturmfels}
Mike Develin, Francisco Santos, and Bernd Sturmfels.
\newblock On the rank of a tropical matrix.
\newblock In {\em Combinatorial and computational geometry}, volume~52 of {\em
  Math. Sci. Res. Inst. Publ.}, pages 213--242. Cambridge Univ. Press,
  Cambridge, 2005.

\bibitem[EH86]{eisenbud-harris}
David Eisenbud and Joe Harris.
\newblock Limit linear series: Basic theory.
\newblock {\em Invent. math.}, 85:337--371, 1986.

\bibitem[Jen14]{jensen}
David Jensen.
\newblock The locus of {Brill-Noether} general graphs is not dense.
\newblock preprint, \arxiv{1405.6338}, 2014.

\bibitem[Kat14]{katz}
Eric Katz.
\newblock Matroid theory for algebraic geometers.
\newblock \textit{Simons Symposium volume}, to appear, 2014.

\bibitem[KP11]{katz-payne}
Eric Katz and Sam Payne.
\newblock Realization spaces for tropical fans.
\newblock In {\em Combinatorial aspects of commutative algebra and algebraic
  geometry}, volume~6, pages 73--88, Berlin, 2011. Abel Symp., Springer.

\bibitem[Laf03]{lafforgue}
Laurent Lafforgue.
\newblock {\em Chirurgie des grassmanniennes}, volume~19 of {\em CRM Monograph
  Series}.
\newblock American Mathematical Society, Providence, RI, 2003.

\bibitem[Len14]{len}
Yoav Len.
\newblock A note on algebraic rank, matroids, and metrized complexes.
\newblock preprint, \arxiv{1410.8156}, 2014.

\bibitem[LV13]{lee-vakil}
Seok~Hyeong Lee and Ravi Vakil.
\newblock Mn{\"e}v-{S}turmfels universality for schemes.
\newblock In {\em A celebration of algebraic geometry}, volume~18 of {\em Clay
  Math. Proc.}, pages 457--468. Amer. Math. Soc., Providence, RI, 2013.

\bibitem[Mn{\"e}88]{mnev}
Nikolai~E. Mn{\"e}v.
\newblock {\em The universality theorems on the classification problem of
  configuration varieties and convex polytopes varieties}, volume 1346 of {\em
  Lect. Notes in Math.}, pages 527--543.
\newblock Springer, Berlin, 1988.

\bibitem[Oxl92]{oxley}
James Oxley.
\newblock {\em Matroid Theory}, volume~21 of {\em Oxford graduate texts in
  mathematics}.
\newblock Oxford Univ. Press, Oxford, 1992.

\end{thebibliography}

\end{document}